\documentclass{article}
\usepackage{bbm}
\usepackage{amssymb}
\usepackage{amsmath}
\usepackage{amsthm}
\usepackage{xcolor}
\usepackage{aligned-overset}
\usepackage{tikz}
\usetikzlibrary{arrows}
\usepackage{tikz-cd}
\usepackage{authblk} 
\usepackage[pdftex,bookmarks,colorlinks]{hyperref}
\usepackage{enumitem}
\hypersetup{colorlinks=false}

\newcommand{\leqnomode}{\tagsleft@true\let\veqno\@@leqno}
\newcommand{\E}{\mathbb{E}}        
\newcommand{\R}{\mathbb{R}}        
\newcommand{\NN}{\mathbb{N}}

\newcommand{\p}{\mathbb{P}}         
\newcommand{\F}{\mathcal{F}}        
\newcommand{\M}{\mathrm{M}}
\newcommand{\1}{\mathbbm{1}}
\DeclareMathSymbol{\shortminus}{\mathbin}{AMSa}{"39}
\newcommand{\dd}{\textnormal{d}} %

\newtheorem{thm}{Theorem}
\newtheorem{lem}[thm]{Lemma}
\newtheorem{prop}{Proposition}
\newtheorem{coro}{Corollary}
\newtheorem{defi}{Definition}
\newtheorem{rmk}{Remark}

	\title{$\Lambda$-Seed-Bank-Wright-Fisher process conditioned on fixation} \date{\today}
	\author[1]{M.~C. Fittipaldi\thanks{MCF's research is supported by the UNAM-PAPIIT grant IN109924.}}
	\author[2]{ A. González Casanova}
	\author[3]{J.~E. Nava Trejo\thanks{JENT's research is supported by Deutsche Forschungsgemeinschaft through IRTG 2544 “Stochastic Analysis in Interaction".}}
	
	\affil[1]{Facultad de Ciencias, UNAM, México}
	\affil[2]{School of Mathematics and Statistical Sciences
		and the Center for Mechanisms of Evolution (Biodesign Institute), Arizona State University, USA}
	\affil[3]{Institut für Mathematik, Humboldt-Universität zu Berlin, Germany}
	
\begin{document}
		\maketitle   
		\begin{abstract}			
			We investigate the $\Lambda$-Seed-Bank-Wright-Fisher process, a model describing allele frequency dynamics in populations exhibiting both skewed offspring distributions and dormancy. By performing a change of measure, we condition this process on the eventual fixation of a specified genetic type. The resulting process is again a $\Lambda$-Wright-Fisher process with a seed bank, but now features coordinated mutations driven by a random switching environment. Our analysis relies on two key techniques: the lookdown construction and sampling duality. These tools provide a pathwise construction of the conditioned process while preserving a means to recover the conditioned population genealogy. The resulting genealogy corresponds to a structured $\Lambda$-coalescent with coordinated mutations determined by the switching environment.	
		\end{abstract}
		\section{Introduction}
		One of the goals of population genetics is to develop models that address key biological questions, for example, how different evolutionary forces shape the diffusion of genes and influence the genealogical structures embedded within a population. A basic question in this context is to describe the frequency dynamics of a particular trait in a haploid population, which leads to one of the central objects in the field: the Wright-Fisher diffusion.  A natural follow-up question is to determine the distribution of gene evolution conditional on the eventual fixation of a specific trait. For the Wright-Fisher diffusion, this can be solved using classical one-dimensional diffusion techniques, as explained later. However, the problem becomes substantially more challenging once additional evolutionary forces acting on the studied population are taken into account. In particular, when the offspring distribution is skewed and the population has a seed-bank, the standard one-dimensional framework is no longer adequate, and more sophisticated probabilistic methods are required. In this work, we employ a lookdown construction to tackle this question and to recover the genealogical structure of the population.  The aim of the paper is to contribute to the existing literature that performs changes of measure of stochastic processes via lookdown constructions \cite{henardChange,K&R}.

		In \cite{FV79}, the Fleming–Viot process was introduced to trace the probability distribution of general traits within a population.  Later, in \cite{D&K_ClassicLD}, Donnelly and Kurtz provided a description of the Fleming-Viot process in terms of an infinite exchangeable particle system, the lookdown construction of the Moran model.  They also established the duality between the Fleming–Viot process and the Kingman coalescent. Further generalizations of the model include populations with skewed offspring distributions, known as the $\Lambda$-Fleming-Viot model, suggested in \cite{DKPRMVPM} and formally introduced in \cite{Bertoin2003}.  The associated genealogical process corresponds to the $\Lambda$-coalescent; see \cite{DKPRMVPM,PitmanLambda99,SagitovLambda99} for details.
		
		Seed banks play a crucial role in buffering against random fluctuations in allele frequencies by storing genetic diversity that can be released into populations during periods of environmental stress.
		Stochastic transitions between dormancy and activity have been identified as a possible cause of highly skewed offspring distributions \cite{ErikWright}. This hypothesis is examined mathematically in \cite{LambdaFromDormancy}, where the authors consider a binary branching population subject to alternating temporal seasons that induce dormancy and regulate population size. The resulting genealogy is characterized by the $\Lambda$-coalescent. A different coordination mechanism for seed-bank populations was proposed in \cite{Coordinatedswitching}, involving simultaneous switching times for dormancy and activation in a Wright-Fisher-type setting. For a broader overview of particle systems with coordination mechanisms, we refer the reader to \cite{CoordinatedParticleSystems}. The seed-bank contributes to the emergence of interesting mathematical phenomena, such as the not coming down from infinity in the seed-bank coalescent, see \cite{BAKW16} and the differences in the structural properties of the seed-bank diffusion compared with the one of the two-island model, see \cite{BBGW19}.
						    	
		As mention before, it is a classical problem to condition diffusion processes on future events, such as the process reaching a specified boundary point. For the Wright-Fisher diffusion $(Z_t)_{t \ge 0}$, defined as the solution of the stochastic differential equation
		\begin{equation} \label{WF_diff}
			dZ_t = \sqrt{Z_t(1-Z_t)}dB_t ,
		\end{equation}
		where $(B_t)_{t\geq 0}$ denotes a standard Brownian motion. By means of the Doob $h$-transform, one can obtain the distribution of the Wright-Fisher diffusion conditioned on eventually reaching $1$, that is, on the event $\lim_{s \to \infty} Z_s = 1$. We denote the resulting process by $(\widetilde{Z}_t)_{t \ge 0}$; it satisfies the stochastic differential equation
		\begin{equation}
			\label{WF_diff}
			d\widetilde{Z}_t = (1-\widetilde{Z}_t)dt+\sqrt{\widetilde{Z}_t(1-\widetilde{Z}_t)}d\widetilde{B}_t ,		
		\end{equation}
		and it is commonly referred to as the Wright-Fisher diffusion with mutation, see \cite[Example 1, Section 9, Chapter 15]{karlin1981second}. This standard approach is sufficient to characterize the distribution of the conditioned process; however, it does not appear, at first glance, to be compatible with the individual-based models typically used to approximate the Wright-Fisher diffusion and, subsequently, to characterize the associated ancestral process. To address this matter, we employ a combination of change-of-measure arguments together with the lookdown construction of the approximating individual based models.
		
		Lookdown type constructions, originally introduced by Donnelly and Kurtz \cite{D&K_ClassicLD,DKPRMVPM}, provide alternative ways to formulate population genetic models in which individuals carry a “level” that orders them. This framework allows one to build, in a monotone way, systems of infinitely many particles. For a modern introduction to the topic, we refer the reader to \cite{E&K,K&R}. In the context of seed-bank models, in  \cite{LookdownSeedBank} we  recently introduced a lookdown construction for the Moran seed-bank model, which proved useful for a more detailed description of the time to the most recent common ancestor in the seed-bank coalescent originally described in \cite{BAKW16}. An intermediate goal of the present paper is to extend this approach by developing a lookdown construction for populations with skewed offspring distributions and seed-banks.
		
		The idea of applying a change of measure at the lookdown level is already present in the literature; two important references to mention are \cite{henardChange,K&R}. In these works, the authors recover classical results such as branching processes conditioned on extinction or non-extinction and indicate how other results, including the immortal particle construction, can be derived from the lookdown framework. In Hénard \cite{henardChange}, the author proposes several change-of-measure techniques that yield additive and multiplicative Doob $h$-transforms of the Fleming-Viot process and the Dawson-Watanabe superprocess, thereby enabling further extensions of the $\Lambda$-Fleming-Viot process. We refer the reader to previous studies on the long-term behavior of measure-valued processes that employ techniques other than the lookdown construction, such as \cite{Overbeck1994}. For a more detailed discussion of backbone decomposition and change-of-measure methods in the context of branching and measure-valued processes, see \cite{Englander&Kyprianou,Fittipaldi&Fontbona,Foutel-Rodier2023}. Alternative graphical constructions have also been developed to encode genealogical information providing further insight into the population structure at fixation, see \cite{Greven2016,POKALYUK201325}.
		
		We place particular emphasis on \cite{henardChange} regarding the multiplicative Doob $h$-transform. Broadly speaking, when aiming to obtain the Fleming-Viot process conditioned on the coexistence of $K$ types, Hénard derives the corresponding change of measure in the lookdown construction by disregarding reproduction events that involve at least two of the $K$ lowest-level individuals. We adapt this idea by designating the individual at the lowest level in the lookdown model as a particular type, and then characterizing how the remainder of the particle system behaves when the active or dormant state of this lowest level individual is also given. This intuition will later be applied to the individual based model used for the diffusive approximation.  
		
		The $\Lambda$-Seed-Bank-Wright-Fisher process is a model designed to capture the dynamics of genetic variation within a population that evolves with a seed-bank and experiences large reproduction events. This model generalizes the classical $\Lambda$-Wright-Fisher framework to include a reservoir of stochastic size. As in \cite{PitmanLambda99}, the sizes of large reproduction events are encoded by $\Lambda$, a measure on $[0,1]$ satisfying the usual integrability condition
		\begin{equation}
			\int_0^1\frac{\Lambda(\dd y)}{y^2}<\infty.
		\end{equation} 
		It is convenient to rewrite $\Lambda$ as 
		\begin{equation}\label{eq: lambda}
			\Lambda(\cdot) = a_0\delta_0(\cdot) + \Lambda_0(\cdot), \quad \text{ where } \Lambda_0(\{0\})=0.
		\end{equation}
		Let $\alpha,\sigma\geq 0$ be the deactivation and activation rates, respectively. In Section 4, we will define the \emph{$\Lambda$-Seed-Bank-Wright-Fisher process} $\mathbf{Z}\coloneqq \left\{\mathbf{Z}(t)\right\}_{t\geq 0}$ as the unique strong solution (see Theorem \ref{thm:Strong_sol}) to the system of SDEs
		\begin{equation} \label{eq:sys_sdbk}  
			\left\{
			\begin{split} 
				Z_1(t)&  =  z_1
				+\int_{0}^t \left(\sigma Z_2(s)-\alpha     Z_1(s)\right) \dd s  \\
				& \quad +\int_{0}^t\sqrt{a Z_1(s)\left(Z_3(s)-Z_1(s)\right)} \dd B_s\\
				& \quad +\int_0^t\int_0^1\int_0^1 r\left[Z_3(s_{\shortminus})-Z_1(s_{\shortminus}) \right]\1_{\left\{u \leq\frac{Z_1(s_{\shortminus})}{Z_3(s_{\shortminus})} 	\right\}}\widetilde{N}(\dd s,\dd r,\dd u)
				\\
				&  \quad - \int_0^t\int_0^1\int_0^1 r Z_1(s)\1_{\left\{u>\frac{Z_1(s_{\shortminus})}{Z_3(s_{\shortminus})}\right\}}\widetilde{N}(\dd s,\dd r,\dd u)\\
				Z_2(t) & =  z_2+ \int_{0}^t \left(\alpha Z_1(s)-\sigma Z_2(s)\right) \dd s\\
				Z_3(t) & =  z_3+\int_0^t \left[\sigma (1-Z_3(t))-\alpha Z_3(t)\right]\dd t, \quad t> 0,
			\end{split}
			\right.
		\end{equation} 
		starting from $\mathbf{Z}(0)= \mathbf{z}$ in $D\coloneqq \left\{\mathbf{z}\in[0,1]^3\,:\, z_1\leq z_3,\,  z_2\leq 1-z_3 \right\}$. Here $\{B_t\}_{t\geq 0}$ is a standard Brownian motion and $\widetilde{N}$ is a compensated Poisson random measure on $\R^+\times (0,1]\times [0,1]$ with intensity measure $\dd t\otimes \Lambda_0(\dd r)r^{-2}\otimes \dd u$ and $f(s_{\shortminus})\coloneqq \lim_{s_n\uparrow s} f(s_n)$ for any function $f$. In the system \eqref{eq:sys_sdbk} for fix $t\geq 0$ the variables $Z_1(t)$ and $Z_2(t)$ describe the frequencies of active and dormant individuals, respectively, for a particular trait, while $Z_3(t)$ denotes the frequency of active individuals in the total population at that time. 
		
		Our main focus is to characterize the model conditioned on the eventual fixation of a given type within the population. To this end, we introduce a diffusion process that incorporates coordinated mutations determined by a switching environment $\xi \coloneqq \{\xi_t\}_{t\geq 0}$, given by a continuous time Markov chain on $\{0,1\}$ with transition rates $q_{01}=\alpha$ and $q_{10}= \sigma$. For this process, called the {\it $(\Lambda, \xi, \mathrm{M})$-Seed-Bank-Wright-Fisher process}, $\mathrm{M}$ represents a measure on the interval $[0,1]$ that encodes the intensity of the mutations. In the same spirit of the reproduction measure, we decompose $\mathrm{M}$ as 
		\begin{equation}\label{eq: M}
			\mathrm{M}(\cdot) = b_0\delta_0(\cdot) + \mathrm{M}_0(\cdot), \quad \text{ with } \mathrm{M_0}(\{0\})=0
		\end{equation} 
		The $(\Lambda,\xi,\M)$-Seed-Bank-Wright-Fisher is then the unique strong solution (see Theorem \ref{thm:Strong_sol}) to  the system of SDEs
		
		\begin{equation} \label{eq:sys_sdbkm}
			\left\{
			\begin{split} 
				\widetilde{Z}_1(t) & =  z_1  + \int_{0}^t \left(\sigma\widetilde{Z}_2(s)-\alpha\widetilde{Z}_1(s)\right) \dd s  \\
				& \quad+\int_{0}^t\sqrt{a \widetilde{Z}_1(s)\left(\widetilde{Z}_3(s)-\widetilde{Z}_1(s)\right)} \dd B_s\\
				& \quad+\int_0^t\int_0^1 \int_0^1 r\left[\widetilde{Z}_3(s_{\shortminus})-\widetilde{Z}_1(s_{\shortminus}) \right] \1_{\left\{u\leq\frac{\widetilde{Z}_1(s)}{\widetilde{Z}_3(s)} \right\}}N(\dd s,\dd r,\dd u)\\
				& \quad-\int_0^t\int_0^1 \int_0^1 r \widetilde{Z}_1(s)\1_{\left\{u>\frac{\widetilde{Z}_1(s)}{\widetilde{Z}_3(s)}\right\}}\widetilde{N}(\dd s,\dd r,\dd u)\\
				& \quad+ \int_0^t b\xi_s(\widetilde{Z}_3(s)-\widetilde{Z}_1(s))\dd s\\
				& \quad+\int_0^t\int_0^1 r\xi_s(\widetilde{Z}_3(s_{\shortminus})-\widetilde{Z}_1(s_{\shortminus}) )\widehat{N}(\dd s,\dd r) \\
				\widetilde{Z}_2(t) & =  z_2  + \int_{0}^t \left(\alpha\widetilde{Z}_1(s)-\sigma\widetilde{Z}_2(s)\right) \dd s\\
				\widetilde{Z}_3(t) & =  z_3  + \int_0^t \left[\sigma\left(1-\widetilde{Z}_3(s)\right)-\alpha\widetilde{Z}_3(s) \right] \dd s,  \quad t>0,
			\end{split}
			\right.
		\end{equation}
		starting from $\mathbf{\widetilde{Z}}(0)= \mathbf{\widetilde{z}} \in D$, where $(B_t)_{t\geq 0}$ is a standard Brownian motion, and $\widetilde{N},\widehat{N}$ are compensated Poisson random measures on $\R^+\times (0,1]\times [0,1]$ with intensities $\dd t\otimes \Lambda_0(\dd r)r^{-2}\otimes \dd u$ and $\dd t\otimes \M_0(\dd r)r^{-1}\otimes \dd u$, respectively. Note that the last line on the right-hand side of the first equation corresponds to the mutations given by the switching environment.

		Fixation in the $\Lambda$-Seed-Bank Wright-Fisher process occurs whenever the frequency of $\heartsuit$-type eventually becomes one, which is formally described by 
		\begin{align*}
			\label{rates:pure_markov_jump}
			\lim_{s\rightarrow\infty} Z_1(s) +Z_2(s) = 1.
		\end{align*}
		We are now prepared to state our main theorem, which establishes a connection between the $\Lambda$-Seed-Bank Wright-Fisher diffusion conditioned on fixation and the corresponding model with mutations governed by a random environment.
		\begin{thm}
			\label{thm:fixation}
			Let $\mathbf{Z}\coloneqq \left\{\mathbf{Z}(t)\right\}_{t\geq 0}$ be a $\Lambda$-Seed-Bank-Wright-Fisher process, define \[\widetilde{\mathbf{Z}} = \left\{\mathbf{Z}(t)\, \Big|\lim\limits_{s\rightarrow \infty}Z_1(s)+Z_2(s)=1\right\}_{t\geq 0}\] then there exist a continuous time Markov chain $\xi$ on $\{0,1\}$, with $q_{01}=\alpha$ and $q_{10}=\sigma$ such that $\widetilde{\mathbf{Z}}$ distributed as a $(\Lambda,\xi,\M)$-Seed-Bank-Wright-Fisher process with switching environment $\xi$ and mutation measure given by $\mathrm{M}(\{0\})=a$ and $\mathrm{M}_0(\dd y) = \Lambda_0(\dd y)y^{-1}$.
		\end{thm}
		
		The remaining sections are organized as follows: In Section \ref{1-MoranModel}, we introduce the individual-based model, along with its forward-in-time frequency process and genealogical structure, linking them via sampling duality. In Section \ref{2.Lookdown}, we formulate the lookdown construction for the Seed-Bank model and show that it induces the same empirical measure and associated frequency process as its respective Moran counterpart. In Section \ref{3. Convergence results}, we establish the existence and uniqueness of solutions to the systems \eqref{eq:sys_sdbk} and \eqref{eq:sys_sdbkm}, as well as the convergence of the frequency processes of the individual-based models to these systems. The convergence theorem will follow from the sampling duality relation \eqref{eq: defH}, since this equation allows us to prove the tightness of the sequence of frequency processes, as stated in Lemma \ref{lem:dual_tight}. This represents a novel use of duality to establish tightness. Finally, in Section \ref{4. Conditioned process}, we use the lookdown construction to reframe the event of fixation in terms of the type of the lower individual. This new perspective induces a clear probabilistic relation between the two main objects discussed in this work, namely \eqref{eq:sys_sdbk} and \eqref{eq:sys_sdbkm}. Sections \ref{proof:Strong_sol}, \ref{app: pf wc}, and \ref{pf: fixation} contain the proofs of Theorem \ref{thm:Strong_sol}, Theorem \ref{thm:weak_conv}, and Theorem \ref{thm:fixation}, respectively.     
		
		\section{Moran model and  Lookdown construction}\label{1-MoranModel}
		In this section, we introduce an individual-based model in the family of Moran models whose frequencies converge to the $(\Lambda,\xi,\M)$-Seed-Bank-Wright-Fisher model. In addition, we describe the ancestry process of the individual-based model. We show the relation of the frequency process of type $\heartsuit$ and the block counting process of the coalescent process via sampling duality, which leads to moment duality in the limiting process.
		
		\subsection{$N$-$\left(\Lambda,\xi,\mathrm{M}\right)$-Seed-Bank-Moran model and its Poissonian construction} \label{sec:PP_construction}
		As mentioned above, it is convenient to decompose the measure $\Lambda$ as in \eqref{eq: lambda}, into \textit{``small offspring events"} and \textit{``large offspring events"}.
		In a similar way, we decompose the mutation measure as in \eqref{eq: M},  into \textit{``single mutations"} and \textit{``coordinated mutations"}. We assume that the following integral condition to hold 
		\begin{align*}
			\int_{0}^{1}\frac{\Lambda (\dd y)}{y^2}<\infty \quad \text{and} \quad \int_{0}^{1}\frac{\mathrm{M} (\dd y)}{y}<\infty.
		\end{align*}
		This conditions occurs naturally in coordinated particle systems, see \cite{CoordinatedParticleSystems}.
		Consider now a haploid population of fixed size $N\in\NN$ that evolves according to the following rules:
		\begin{itemize}
			\item Each individual has an initial type from $E=\{\heartsuit,\spadesuit\}$ and an initial state $\{\mathrm{a,d}\}$, which correspond to active and dormant, respectively, according with an exchangeable distribution on $\left(E\otimes\{\mathrm{a,d}\} \right)^{\otimes N}$.
			\item \emph{Activation and deactivation events.} The life of each individual consist in successive active and dormant periods that are exponentially distributed with rates $\alpha>0$ and $\sigma>0$, respectively.
			\item \emph{Small reproduction events:} Each unordered pair of active individuals reproduces at rate $a_0 = \Lambda(\{0\})$. A parent is chosen randomly from the pair, and the other individual is replaced by a copy of the parent.
			\item \emph{Large reproduction events:} In each event, every active individual independently decides to participate in the reproduction event according to a Bernoulli with (common) success probability $\Lambda_0(\dd y)$. The parent is randomly chosen among all participants, and the others are replaced by a copy of the parent.
			\item \emph{Environmental switching}. Let $\xi = \{\xi_t\}_{t\geq 0}$ be a continuos Markov chain on $\{0,1\}$ with transition rates given by $q_{01}=\alpha, q_{10}=\sigma$. When this environment process is one, the environment allows mutation, otherwise no mutations occurs.
			\item \emph{Single mutation to type $\heartsuit$ :} Given that the environment allows mutations, $\xi_t = 1$, each active individual independently mutates to type $\heartsuit$  at rate $b_0$. 
			\item \emph{Coordinated mutation to type $\heartsuit$ :} Given that the environment allows mutations, $\xi_t = 1$, in each event, given by a point $(t,y)$ for the Poisson Point Process $\mathrm{M}$, every active individual independently mutates according to a Bernoulli with the same success probability $y$ (otherwise it stays the same).  \\
		\end{itemize}		
		
		\begin{figure}[h]
			\centering
			\includegraphics[width=0.7\textwidth]{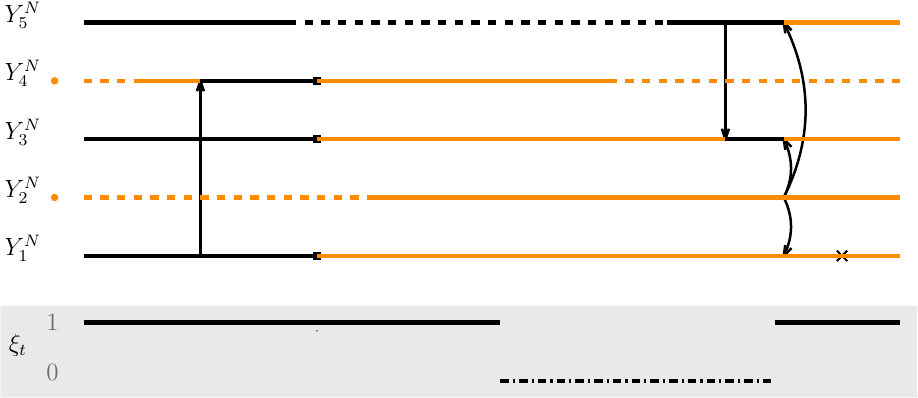}
			\caption{$N$-$(\Lambda,S,\mathrm{M})$-Moran Model. Solid lines represent periods of activity, while dashed lines indicate periods of dormancy. Single arrows represent small reproduction events, whereas multiple arrows indicate large reproduction events. The environment marks the times during which mutations are possible ($\xi_t = 1$). Individual mutations are marked with a cross, and coordinated mutations are indicated by squares.}
		\end{figure}
		
		Therefore, the {\it $N$-$(\Lambda,\xi,\mathrm{M})$-Seed-Bank-Moran model} is written as a random vector on $\left(E\otimes\{\mathrm{a,d}\}\right)^{\otimes N}$
		\begin{equation*}
			\mathbf{Y}^N(t) =  \left( \left(Y_1^{N,1}(t),Y_1^{N,2}(t)\right),\ldots,\left(Y_N^{N,1}(t),Y_N^{N,2}(t)\right) \right), \, \forall t\geq 0,
		\end{equation*}
		where $Y_i^{N,1}(t)$ and $Y_i^{N,2}(t)$ correspond to the type and the state of the $i$-th individual at time $t$, respectively. The model without mutation, where $\mathrm{M}([0,1]) = 0$, is referred to as the {\it $\Lambda$-Seed-Bank-Moran model}.
		
		The Poissonian construction of the $N$-$(\Lambda,S,\mathrm{M})$-Moran model can be done as follows. The activity and inactivity periods of each individual correspond to a continuous time Markov chain on $\{\mathrm{a,d}\}$ with transition rates $\alpha$ and $\sigma$ respectively. On top of these state processes, reproduction and mutation events are built, using two families of Poisson Processes $\{\mathfrak{N}_{ij}\}_{i,j\in\NN,i< j}$ and $\{\mathfrak{N}_{i}\}_{i\in\NN}$, with rates $a_0$ and $b_0$ respectively; and two Poisson Point Processes $\mathfrak{M}^{\Lambda_0}$ and $\mathfrak{M}^{\mathrm{M}_0	}$  on $\R_+\times [0,1] \times [0,1]^\NN$ with intensity measures 
		\begin{equation*}
			\dd t\otimes y^{-2}\mathrm{\Lambda_0}(\dd  y)\otimes \left(\mathbbm{1}_{[0,1]}(t)\dd t\right)^{\otimes\NN} \quad \text{and} \quad 	\dd t\otimes y^{-1}\mathrm{M}(\dd y)\otimes \left(\mathbbm{1}_{[0,1]}(t)\dd t\right)^{\otimes\NN}
		\end{equation*}
		respectively. The Poisson Process $\mathfrak{N}_{ij}$ indicates the possible reproduction events between individuals $i$ and $j$, which only occur if both are active. On the other hand, large offspring events are encoded by $\mathfrak{M}^{\Lambda_0}$. Given an atom $(t,y,(u_i)_{i\in \NN})$ of~$\mathfrak{M}^{\Lambda_0}$, individual $i\in [N]$ participates in the reproduction if $u_i < y$ and $Y_i^{N,2}(t) = \mathrm{a}$ (otherwise it does nothing). Possible single mutations of individual $i$ are dictated by $\mathfrak{N}_{i}$, and are successful only if $\xi_t=1$ and $Y_i^{N,2}(t) = \mathrm{a}$. Coordinated mutations are governed by the Point Process $\mathfrak{M}^{\mathrm{M}}$, that is for $(t,y,(u_i)_{i\in \NN})$ an atom of $\mathfrak{M}^{\mathrm{M}}$ we accept the mutation if $\xi_t = 1$, and then each individual $i\in [N]$ mutates to type $\heartsuit$  if $u_i < y$ and $Y_i^{N,2}(t) = \mathrm{a}$ (otherwise nothing happens to that individual). 
		\begin{rmk}
				Given an initial exchangeable configuration $Y^N(0)$, the particle system $Y^N(t)$ is exchangeable for any $t\geq 0$.
		\end{rmk}
		
		\subsection{Frequency processes}
		In order to trace the frequency of type one individuals within the $N$-$(\Lambda,\xi,\M)$-Seed-Bank-Wright-Fisher model it is necessary to follow the individuals of type one in the active and inactive sub-populations, as well as the total of active individuals of the population. To this end,  consider 
		\[
		\left\{\mathbf{Z}^N(t)\right\}_{t\geq 0}  = \left\{\left(Z_1^N(t),Z_2^N(t),Z_3^N(t)\right)\right\}_{t\geq 0}
		\]
		given by	
		\begin{equation*}
			\left\{
			\begin{split}
				&Z_1^N(t) \coloneqq \frac{1}{N}\sum\limits_{i=1}^N \1_{\{\mathbf{Y}_i^N(t)=(1,\mathrm{a})\}} \\
				&Z_2^N(t) \coloneqq \frac{1}{N}\sum\limits_{i=1}^N \1_{\{\mathbf{Y}_i^N(t)=(1,\mathrm{d})\}} \\
				&Z_3^N(t)\coloneqq \frac{1}{N}\sum\limits_{i=1}^N\sum\limits_{e\in E} \1_{\{\mathbf{Y}_i^N(t)=(e,\mathrm{a})\}} ,
			\end{split}
			\right.
		\end{equation*}
		where $Z_1^N$, $Z^2_N$ and $Z^3_N$ corresponds to the frequency of active type $\heartsuit$ individuals,  the frequency of inactive type $\heartsuit$ individuals and the frequency of active individuals, respectively .
		Therefore, the process $\left\{\left(\mathbf{Z}^N(t),\xi_t\right)\right\}_{t\geq 0}$ is a Markov process with generator
		\begin{equation}\label{gen::freq_n}
			\begin{split}
				\mathcal{G}^Nf(\mathbf{z},s)&  \coloneqq   
				\mathcal{G}_{WF}^Nf(\mathbf{z},s) 
				+ \mathcal{G}_{\Lambda}^Nf(\mathbf{z},s)
				+\mathcal{G}_{\mathrm{M}}^Nf(\mathbf{z},s)
				+ \mathcal{G}_{SB}^Nf(\mathbf{z},s)\\
				& \quad +\mathcal{G}_{\xi}^Nf(\mathbf{z},s),
			\end{split}
		\end{equation}
		that acts on $f\in C_b(D_N\otimes\{0,1\} )$, where 
		\[
		D_N \coloneqq \left\{(z_1,z_2,z_3)\in\frac{[N]_0^{\otimes 3}}{N}\,|\, z_1\leq z_3,\,  z_2\leq 1-z_3 \right\}.
		\]
		In the expression \eqref{gen::freq_n}, we have that the first term in the right hand side correspond to
		\begin{equation} \label{gen_wf_N}
			\begin{split}
				\mathcal{G}_{WF}^Nf(\mathbf{z},s) &\coloneqq   \frac{a_ 0z_1 N^2(z_3-z_1)}{2}\left(f\left(\mathbf{z} + \tfrac{\mathbf{e}_1}{N},s\right)-f(\mathbf{z},s)\right)\\
				& \quad + \frac{a_0z_1 N^2(z_3-z_1)}{2}\left(f\left(\mathbf{z} - \tfrac{\mathbf{e}_1}{N},s\right)-f(\mathbf{z},s)\right),
			\end{split}
		\end{equation}
		where $\mathbf{e}_i$ is the correspondent canonical vector in $\R^3$; the second term is given by
		\begin{equation} \label{gen_lambda_N}
			\begin{split}
				\mathcal{G}_{\Lambda}^Nf(\mathbf{z},s) &\coloneqq Nz_1\sum_{i=2}^{N(z_3 - z_1)}\int_0^1{N(z_3-z_1) \choose i} y^i(1-y)^{N(z_3-z_1)-i}\left(f\left(\mathbf{z} + \tfrac{i\mathbf{e}_1}{N},s\right)-f(\mathbf{z},s)\right) \tfrac{\Lambda_0(\dd y)}{y^2}
				\\
				& + N(z_3-z_1)\sum_{i=2}^{Nz_3} \int_0^1\binom{Nz_1}{i} y^i(1-y)^{Nz_1}\left(f\left(\mathbf{z}-\tfrac{i\mathbf{e}_1}{N},s\right)-f(\mathbf{z},s)\right) \tfrac{\Lambda_0(\dd y)}{y^2};
			\end{split}	
		\end{equation}
		the third term is 
		\begin{equation} \label{gen_M_N}
			\begin{split}
				\mathcal{G}_{\mathrm{M}}^Nf(\mathbf{z},s) &\coloneqq   b_0sN(z_3-z_1)\left(f\left(\mathbf{z} + \tfrac{z_1+\frac{1}{N},z_2,z_3}{\mathbf{e}_1}{N},s\right)-f(\mathbf{z},s)\right)\\
				& + \sum_{i=1}^{N(z_3-z_1)}\int_0^1\binom{N(z_3-z_1)}{i} y^i(1-y)^{N(z_3-z_1)-i}s\left(f\left(\mathbf{z} + \tfrac{i\mathbf{e}_1}{N},s\right)-f(\mathbf{z},s)\right) \tfrac{\mathrm{M}_0(\dd y)}{y} 
			\end{split};
		\end{equation} 
		the forth term is given by
		\begin{equation}\label{gen_SB_N}
			\begin{split}
				\mathcal{G}_{SB}^Nf(\mathbf{z},s) & = \alpha Nz_1\left(f\left(\mathbf{z}-\tfrac{\mathbf{e}_1}{N}+\tfrac{\mathbf{e}_2}{N} - \tfrac{\mathbf{e}_3}{N},s\right)-f(\mathbf{z},s) \right) \\
				& \quad +\sigma Nz_2\left(f\left(\mathbf{z}+\tfrac{\mathbf{e}_1}{N}-\tfrac{\mathbf{e}_2}{N} + \tfrac{\mathbf{e}_3}{N},s\right)-f(\mathbf{z},s) \right)\\
				&\quad +\alpha N(z_3-z_1) \left(f\left(\mathbf{z}-\tfrac{\mathbf{e}_3}{N},s\right)-f(\mathbf{z},s) \right)\\
				& \quad +\sigma N(1-(z_3-z_1) )\left(f\left(\mathbf{z}+\tfrac{\mathbf{e}_3}{N},s\right)-f(\mathbf{z},s) \right)	;
			\end{split}
		\end{equation} and finally 
		\begin{equation}
			\mathcal{G}_{\xi}^Nf(\mathbf{z},s) \coloneqq (\sigma s +\alpha(1-s) )\left(f(\mathbf{z},1-s)-f(\mathbf{z},s)\right).
		\end{equation}
		
		\subsection{Ancestry process}
		In this section we describe the genealogy of a sample in the $N$-$\left(\Lambda,\xi,\mathrm{M}\right)$-Moran model. For $m\leq N$, let $\mathcal{P}_m$ be the set of partitions of  $[m]=\{1,\ldots,m\}$, and define the set of marked partitions as 
		\begin{align*}
			\mathcal{P}_m^{\{\mathrm{a,d}\}} =\left\{(\zeta,l)\Big|\zeta\in \mathcal{P}_m,\, l\in \{\mathrm{a,d}\}^{\otimes{|\zeta|}}\right\}.
		\end{align*}
		As usual in coalescent theory, marked spaces can be understood as each block $\zeta$ having a flag $l$. Now, let's define the operators over the marked partitions.	Let $\pi,\pi'\in \mathcal{P}_m^{\{\mathrm{a,d}\}}$, we say that $\pi\succ_i\pi'$ if $\pi'$ can be obtained from $\pi$ by merging exactly $i$ blocks with $\mathrm{a}$-flag and the resulting block has a $\mathrm{a}$-flag, we say that $\pi \Join \pi'$ if $\pi$ can be obtained from $\pi'$ by changing the flag of precisely one block. Finally, we say $\pi \Join_i \pi'$ if $\pi'$ can be obtained from $\pi$ by deleting exactly $i$ blocks with $\mathrm{a}$-flag.
		\begin{defi} [$k$-$(\Lambda,\xi,\mathrm{M})$-Seed-Bank Coalescent] We define the \textit{$k$-$(\Lambda,\xi,\mathrm{M})$-Seed-Bank Coalescent} to be the continuous time Markov chain with values on $\mathcal{P}_k^{\{\mathrm{a,d}\}}\times \{0,1\}$ with the following transition rates. If  $n,m$ is the number of active and inactive blocks in $\pi$ respectively the transitions $(\pi,s)\mapsto(\pi',s')$ occur at the following rates:
			\begin{equation*}
				\left\{
				\begin{array}{cc}
					a_0\1_{\{i=2\}}+\int_0^1{n\choose i}y^{i-2}(1-y)^{n-i}\Lambda_0(\dd y)& \text{ if } \pi\succ_i \pi' \text{ and } s=s' \text{ for } 2\leq i\leq n\,;\\
					\sigma m & \text{ if } \pi \Join \pi',\, s=s' \text{ and a $d$-flag is replaced by a $a$-flag };\\
					\alpha n & \text{ if } \pi \Join \pi',\, s=s' \text{ and a $a$-flag is replaced by a $d$-flag };\\
					s\left(b_0\1_{\{i=1\}}+ \int_0^1{n\choose i}y^{i-1}(1-y)^{n-i}\mathrm{M_0}(\dd y)\right) & \text{ if }  \pi\Join_i \pi' \text{ and } s=s',\,\text{ for }1\leq i \leq n\, ; \text{ and }\\
					\sigma s +\alpha(1-s) & \text{ if } \pi = \pi' \text{ and } s'=1-s. 
				\end{array}
				\right.
			\end{equation*}
		\end{defi}

		Similarly, the block counting process of the ancestry process of the $(\Lambda,\xi,\mathrm{M})$-Moran model is denoted by $\{(N_t,M_t,\xi_t)\}_{t\geq 0}$, and it is the Markov process where  the transitions $(n,m,s)\mapsto(n',m',s')$ occur at the following rates
		\begin{equation*}
			\left\{
			\begin{array}{cc}
				a\1_{\{i=2\}}+\int_0^1{n\choose i}y^{i-2}(1-y)^{n-i}\Lambda_0(\dd y),\, & \text{ if } n'=n-(i-1), m'=m \text{ and } s'=s, \text{ for } 2\leq i \leq n;\\ 
				\sigma m & \text{ if } n'=n+1, m'=m-1 \text{ and } s'=s;\\
				\alpha n & \text{ if } n'=n-1, m'=m+1 \text{ and } s'=s;\\
				s\left(b\1_{\{i=1\}}+ \int_0^1{n\choose i}y^{i-1}(1-y)^{n-i}\mathrm{M}(\dd y)\right) & \text{ if } n'=n-i, m'=m and s'=s, \text{ for } 1\leq i\leq n;\\
				\sigma s +\alpha(1-s) & \text{if } n'=n, m'=m \text{ and } s'=1-s. 
			\end{array}
			\right.
		\end{equation*}
		
		 We conclude this section by examining the duality relation that emerge between the frequency and genealogy processes associated with the $N$-$(\Lambda, \xi, \mathrm{M})$-Moran model. Define $H^N : D_N \times [N]^2 \rightarrow [0,1]$ as
	\begin{equation}\label{eq: defH}
		H^N(\mathbf{z},n,m) \coloneqq \dfrac{{Nz_1 \choose n}\binom{Nz_2}{m}}{\binom{Nz_3}{n}\binom{N(1-z_3)}{m}}\mathbbm{1}_{\{ z_3 \notin \{0,1\}\}} + \dfrac{\binom{Nz_1}{n}}{\binom{N}{n}}\mathbbm{1}_{\{ z_3 =1\}} +
		\dfrac{\binom{Nz_2}{m}}{\binom{N}{m}} \mathbbm{1}_{\{ z_3 =0\}}
	\end{equation}
		\begin{prop}\label{sampling duality}
			The frequency process $ \{\left(\mathbf{Z}^N(t),\xi_t\right)\}_{t\geq 0 }$ and the ancestry processes  $\{(N_t,M_t,\xi_t)\}_{t\geq 0 }$ associated to a $N$-$(\Lambda,\xi,\mathrm{M})$-Moran model are sampling duals. Indeed, we have that
			\begin{align}
				\label{duality:sampling}
				\E_{(z_1,z_2,z_3)}[H^N(\mathbf{Z}^N(t),n,m)] = \E_{(n,m)}[H^N(\mathbf{z},N_t,M_t)]
			\end{align}
			for each $t\geq 0$, $\mathbf{z}\in D^N$ and $n,m\in [N]$.
		\end{prop}
		\begin{proof}
			The sampling question associated to the duality relationship is given by the event
			\begin{center}
				B= {\it ``$n$ active individuals and $m$ inactive individuals sampled uniformly at time $t$ are of type $\heartsuit$,  given that the starting frequencies of active type $\heartsuit$, inactive type $\heartsuit$ and active individuals are $z_1,z_2,z_3$ respectively."}
			\end{center}
			Using a backward and forward approach we deduce that 
			\begin{equation*}
				\p_{(n,m)}[B|\mathbf{Z}^N(t)]  = H^N(\mathbf{Z}^N(t),n,m) \text{ and }  
				\p_{\mathbf{z}}[B|N_t,M_t]  = H^N(\mathbf{z},N_t,M_t).
			\end{equation*} 
		\end{proof}
		
		\section{Lookdown construction for $N$-$(\Lambda,\xi,\mathrm{M})$-Seed-Bank-Moran model}\label{2.Lookdown}
		\subsection{Ordered model} \label{subsec: ormodel}
		In this section, we define an alternative ordered model with the same frequency law as the Moran model, which has the advantage of being monotone in the sense of the lookdown construction introduced by \cite{D&K_ClassicLD}. We revisit the classic lookdown constructions presented in \cite{D&K_ClassicLD} to introduce an ordered model, referred as $(\Lambda,\xi,\mathrm{M})${\it-Seed-Bank-lookdown model}. For a population of size $N$, the distinction between the ordered and unordered $N$-particle models lies in the mechanism for selecting the parent during a reproduction event. In the ordered model, the parent is the  individual in the lowest level among all those involved. We will denote the type and state of the ordered particle system at time $t$ as
		\begin{equation*}
			\begin{split}
				\mathbf{X}^N(t)&\coloneqq(\mathbf{X}_1^N(t),\ldots,\mathbf{X}_N^N(t))\\
				&= \left( \left(X_1^{N,1}(t),X_1^{N,2}(t)\right),\ldots,\left(X_N^{N,1}(t),X_N^{N,2}(t)\right) \right).
			\end{split}
		\end{equation*}
		The Poisson processes outlined in Section~\ref{sec:PP_construction} provide the means to formulate a $(\Lambda,\xi,\mathrm{M})$-Seed-Bank-lookdown model simply by modifying the parent selection mechanism. Moreover, the Poissonian construction will be crucial to establish the connection between ordered and unordered model.
		
		\begin{thm}\label{thm:ld_exch}
			If the initial distribution $\mathbf{X}^N(0)$ is exchangeable, then the $(\Lambda,\xi,\mathrm{M})$-Seed-Bank-lookdown model $\mathbf{X}^N(t)$ is exchangeable for each $t\geq 0$. 
		\end{thm} 
		
		To prove this result we adapt arguments of \cite{JB2008_XI_FV,D&K_ClassicLD,DKPRMVPM}. That is, we construct a coupling of the two models via a process taking values in the permutations of $[N]$, $\Theta=\{\Theta(t)\}_{t\geq 0}$, such that at any time $t$, the $\Theta_t$ is uniformly distributed among all permutations of $[N]$,
		\begin{align*}
			\left(Y_1(t),\ldots,Y_N(t)\right) = \left(X_{\Theta_1(t)}(t),\ldots,X_{\Theta_N(t)}(t)\right),\quad t\geq 0,
		\end{align*}
		and $\Theta$ is independent of the $\sigma$-algebra generated by the "demographic information" \eqref{eq:dem_info}.
		
		Let us construct the permutation process.
		To simplify the notation, we will eliminate the superscript $N$ in the following. Since both process behave in the same way in the mutation events, it is sufficient to construct $\{\Theta(t)\}_{t\geq 0}$ in terms of its skeleton chain $\{\theta_m\}_{m\in\NN}$ over the random times $\{t_m\}_{m\in\NN}$ when reproduction events occurs. For $t\geq 0$, define the set of active individuals at time $t$ as 	
		\begin{equation}
			A^N(t) = \{i\in [N]| X_{i}^{N,2}(t)=\mathrm{a}\} 
		\end{equation}
		and $\mathcal{P}_{A,k}$ the collection of all subsets of $A\subseteq [N]$ of cardinality $k$. 
		
		First, select $\theta_0$ uniformly over $S_N$, the set of all permutations of $[N]$. We will now proceed inductively, suppose the reproduction event $m$-th involves $k$ individuals. Let $\phi_m$ be the set of levels of involved individuals. Observe that $\phi_m$ correspond to an element of $P_{A^N(t_m),k}$ selected uniformly over it, $\phi_m$ is independent of $\theta_{m-1}$; and $\psi = \theta_{m-1}^{-1}(\phi_m)$ correspond to the disordered particle system levels involved in the $m$-th reproduction event. The next step corresponds to a random rearrangement of $\psi_m$ in the permutation process, let $\sigma_m \in S_k$ and set 
		\begin{align*}
			\theta_m(\psi_m(i)) = \phi_m(\sigma_m(i)), \quad i \in [k].
		\end{align*}
		the rest of the permutation is left as in the previous step it is to say $\theta_m(j) =\theta_{m-1}(j)$ for $j \in [N]\setminus \psi_m$. The reproduction information of the $m$-th reproduction event, the parent particle and its respective child can be encoded by 
		\begin{align*}
			\chi_m = (\psi_m(1),\psi_m \setminus \psi_m(1)).
		\end{align*}
		The following lemma establishes that the reordering of the levels is independent of the demographic history.
		\begin{lem}
			\label{lem_unif_perm}
			For each $m\in \NN$, $\chi_1,\ldots,\chi_m,\theta_m$ are independent. Furthermore $\theta_m$ is uniformly distributed over $S_N$.
		\end{lem}
		\begin{proof}
			This result is proven by induction, following the standard approach of  \cite{JB2008_XI_FV,DKPRMVPM}. Define $\F_n = \sigma\left(\theta_m,\chi_m\,|\, 0 \leq m \leq n\right)$. We have that for any $f: S_n\times \left([N]\times \mathcal{P}^N\right)\rightarrow \R $
			\begin{align}
				\label{eq:theta_chi_m}
				\E\left[f(\theta_m,\chi_m)|\F_{m-1} \right] = \E[f(\theta_m,\chi_m)|\theta_{m-1}],
			\end{align}
			since $\theta_m$ and $\chi_m$ are only constructed from $\theta_{m-1}$ and other independent random information. Then for any $f:S_N \rightarrow \R$ and $h_k:[N]\times P_N \rightarrow \R$ 
			\begin{align*}
				\E\left[f(\theta_m) \prod_{k=1}^m h_k(\chi_k) \right] &= \E\left[\E\left[f(\theta_m)h_m(\chi_m)|\F_{m}\right] \prod_{k=1}^{m-1} h_k(\chi_k) \right]\\
				\overset{\eqref{eq:theta_chi_m}}&{=} \E\left[\E\left[f(\theta_m)h_m(\chi_m)|\theta_{m-1}\right] \prod_{k=1}^{m-1} h_k(\chi_k) \right]\\
				& = \E\left[f(\theta_m)h_m(\chi_m)\right]\prod_{k=1}^{m-1} \E\left[h_k(\chi_k)\right]
			\end{align*}
			where we use the induction hypothesis in the last equality. The remaining part of the proof is to show that $\theta_m$ and $\chi_m$ are independent. Note that $\theta_m$ is uniformly distributed since we only rearrange the inputs $\psi_m$ in an uniform manner independent of $\theta_{m-1}$, more over notice that 
			\begin{align*}
				\p\left[\theta_{m-1} = \widetilde{\theta} | \chi_m= \chi, \theta_{m-1} = \widehat{\theta}\right] = \begin{cases}
					\frac{1}{k!} & \text{ if } \widetilde{\theta}(i) = \widehat{\theta}(i), \text{for all } i\in [N]\setminus \chi \\
					0 & \text{ otherwise}
				\end{cases},
			\end{align*} 
			for $\widehat{\theta},\widetilde{\theta} \in S_N$ and $\chi$ an admissible demographic configuration. This implies that 
			\begin{align*}
				\p\left[\theta_{m} = \widetilde{\theta} \Big| \chi_m= \chi\right] & = \sum_{\widehat{\theta} \in S_N} \p\left[\theta_{m} = \widetilde{\theta}, \theta_{m-1} = \widehat{\theta} \Big| \chi_m= \chi\right]\\
				& = \sum_{\widehat{\theta} \in S_N} \p\left[\theta_{m} = \widetilde{\theta}\Big| \theta_{m-1} = \widehat{\theta},\chi_m= \chi\right]\p\left[\theta_{m-1} = \widehat{\theta}\Big|\chi_m= \chi \right] \\
				& = \sum_{\widehat{\theta} \in S_N} \p\left[\theta_{m} = \widetilde{\theta}\Big| \theta_{m-1} = \widehat{\theta},\chi_m= \chi\right]\p\left[\theta_{m-1} = \widehat{\theta}\right]\\
				& = \sum_{\widehat{\theta} \in S_N} \frac{\1_{\left\{\widetilde{\theta}(i) = \widehat{\theta}(i), \forall i\in [N]\setminus \chi\right\} } }{k!}\cdot \frac{1}{N!} = \frac{1}{N!}.
			\end{align*}
			Since the law of $\theta_m$ and $\theta_m$ given $\chi_m$ coincide we conclude that $\theta_m$ and $\chi_m$ are independent, which completes the proof.
		\end{proof}
		\begin{proof}[Proof of Theorem \ref{thm:ld_exch}.]
			Given a realization $\left\{\left(X_1(t),\ldots,X_N(t)\right)\right\}_{t\geq 0}$ of a $(\Lambda,\xi,\mathrm{M})$-Seed-Bank-lookdown model let $\{t_m\}_{m\in\NN}$ be the times in which reproduction events happen. Therefore taking an initial random permutation $\theta_0$ independent of the process, we inductively construct the permutation process $\{\theta_m\}_{m\in\NN}$ as described above. Then, define $\{\Theta(t)\}_{t\geq0}$ as  
			\begin{align*}
				\Theta(t)\coloneqq \theta_m,\quad \text{ for } t_m\leq t<t_{m+1}. 
			\end{align*}
			Then 
			\begin{align*}
				(Y_1(t),\ldots,Y_N(t))\coloneqq \left(X_{\Theta_1(t)}(t),\ldots,X_{\Theta_N(t)}(t)\right),\quad t\geq 0,
			\end{align*}
			correspond to a version of a $(\Lambda,\xi,\mathrm{M})$-Seed-Bank-Moran model, by Lemma~\ref{lem_unif_perm}. Notice that the type process at time $t\geq 0$, $(Y_1^{E,N}(s),\ldots,Y_N^{E,N}(s))$, depends only on the initial type configuration $(Y_1^{E,N}(0),\ldots,Y_N^{E,N}(0))$ the reproduction events $\{\chi_m\}_{m\in \NN,\,t_{m}\leq t}$ taking place before $t$, and the coordinated mutations between reproduction events. Using Lemma~\ref{lem_unif_perm} we can ensure that $\Theta(t)$ and $\Theta^{-1}(t)$ is independent of the ``demographic information" given by 
			\begin{align}
				\label{eq:dem_info}
				\mathcal{G}_t= \sigma\{\mathbf{Y}^N(s)  \,|\,s\leq t\}\vee \sigma\{\chi_m\,|\,m\in\NN\}.
			\end{align}
			Finally, the exchangeability follows from the fact that
			\begin{align*}
				\left(X_1(t),\ldots,X_N(t)\right)=\left(Y_{\Theta_1^{-1}(t)}(t),\ldots,Y_{\Theta_N^{-1}(t)}(t)\right) ,\quad t\geq 0.
			\end{align*}  
		\end{proof}
		\begin{rmk}
			Note that Theorem \ref{thm:ld_exch} is also valid in a strong sense, taking stopping times instead of deterministic times.
		\end{rmk}
		From Theorem \ref{thm:ld_exch} we get the following corollary.
		\begin{coro}
			Starting from the same exchangeable initial condition, the laws of the empirical measures of the $(\Lambda,\xi,\mathrm{M})$-Seed-Bank-Moran model and the $(\Lambda,\xi,\mathrm{M})$-Seed-Bank-look-down model coincide. Similarly, the frequencies of the look-down process coincide with the ones of the $N$-$(\Lambda,\xi,\mathrm{M})$-Seed-Bank Moran Model.
		\end{coro}
		\begin{rmk}
			The lookdown construction can also be derived using the Markov Mapping Theorem \cite{E&K,K&R}, as demonstrated in \cite{LookdownSeedBank}. The strategy presented here allows for an explicit coupling between the ordered and unordered models, a key feature that will be useful in the proofs presented in later sections (see Lemma \ref{lem:fix_levels}, Lemma \ref{lem_fixation}, and the proof of Theorem \ref{thm:fixation}).
		\end{rmk}
		\section{Convergence results}	\label{3. Convergence results}	
		We start this section by showing that the system of equations \eqref{eq:sys_sdbkm} has an unique strong solution $\mathbf{Z}$.  We refer to this solution $\mathbf{Z}$  as the $\left(\Lambda,\xi,\mathrm{M}\right)$-Seed-Bank-Wright-Fisher.
		\begin{thm}
			\label{thm:Strong_sol}
			Given the initial condition $\mathbf{Z}(0)=\mathbf{z}_0\in D$, the system of stochastic differential equations \eqref{eq:sys_sdbkm} has a unique strong solution $\mathbf{Z}$.
		\end{thm}
		\begin{rmk}
			\label{rem:conv_generator}
			As a direct implication of the Taylor expansion of a given function, it is easy to see that the generator given in \eqref{gen::freq_n} converges  to
			\begin{equation}
				\label{gen::freq_l_m}
				\begin{split}
					\mathcal{G}f(\mathbf{z},s) & =\frac{a}{2}z_1(z_3-z_1)f_{z_1z_1}(\mathbf{z},s)\\
					& \quad +z_1\int_{0}^{1}\left(f\left(z_1+y(z_3-z_1),z_2,z_3,s\right)-f(\mathbf{z},s)\right)\tfrac{\Lambda_0(\dd y)}{y}\\
					&\quad + (z_3-z_1)\int_{0}^{1}\left(f\left(z_1-yz_1,z_2,z_3,s\right)-f(\mathbf{z},s)\right)\tfrac{\Lambda_0(\dd y)}{y}\\
					&\quad + \left(\sigma z_2-\alpha z_1 \right)f_{z_1}(\mathbf{z},s)
					+\left(\alpha z_1-\sigma z_2 \right)f_{z_2}(\mathbf{z},s)\\
					& \quad+ \left(\sigma(1-z_3)-\alpha z_3\right)f_{z_3}(\mathbf{z},s)
					+bs(z_3-z_1)f_{z_1}(\mathbf{z},s) \\
					&\quad + s\int_{0}^{1}\left(f\left(z_1+y(z_3-z_1),z_2,z_3,s\right)-f(\mathbf{z},s)\right)\mathrm{M_0}(\dd y)\\
					&\quad  +\left(\sigma s+ \alpha(1-s)\right)\left(f(\mathbf{z},1-s)-f(\mathbf{z},s)\right) 
				\end{split}
			\end{equation}
			as $N$ goes to infinity, which correspond to the generator of the $\left(\Lambda,\xi,\M\right)$-Wright-Fisher process defined above. Here, $f\in [0,1]^{\otimes 3}\times \{0,1\}$  is a function that is twice differentiable in the first entry and once differentiable in the second and third entry.
		\end{rmk} 
		The proof of Theorem \ref{gen::freq_l_m} is in Subsection \ref{proof:Strong_sol}. In the following, we state the convergence in distribution of the frequencies process of the $N$-$(\Lambda,\xi,\mathrm{M})$-Seed-Bank-Moran Model to the $(\Lambda,\xi,\mathrm{M})$-Seed-Bank-Wright-Fisher process.
		
		\begin{thm}
			\label{thm:weak_conv}
			The sequence of frequencies processes $\left\{\mathbf{Z}^N\right\}_{N \in \NN}$, related to a $N$-$(\Lambda,\xi,\mathrm{M})$-Seed-Bank-Moran Model, converges weakly on $D(\R^+,D\otimes \{0,1\})$ to the $(\Lambda,\xi,\mathrm{M})$-Seed-Bank-Wright-Fisher process $\mathbf{Z}$. 
		\end{thm}
		The proof of this result is provided in Section \ref{app: pf wc}. We will adopt the following strategy: first, we will establish the convergence of the sequence of processes by verifying tightness. Then, we will characterize the limiting process by formulating a corresponding martingale problem.
		To carry out the first part of the proof, we will prove the following lemma.
		\begin{lem}
			\label{lem:tight_z}
			The sequence $\left\{\left(\mathbf{Z}^N(t),\xi_t\right)_{t\geq 0}\right\}_{N\in\NN}$ satisfy the Aldous tightness criterion.
		\end{lem}
		Thanks to the duality relation established in Proposition \ref{sampling duality}, this lemma is proved as a consequence of the following result.
		\begin{lem} 
			\label{lem:dual_tight}
			The sequences 
			\[
			\left\{\left\{H^N\left(\mathbf{Z}^N(t),1,0\right)\right\}_{t\geq 0 }\right\} _{N\in \NN} \quad  \text{ and } \quad  \left\{\left\{H^N\left(\mathbf{Z}^N(t),0,1\right)\right\}_{t\geq 0 }\right\} _{N\in \NN}
			\]
			satisfy the Aldous tightness criterium.
		\end{lem}
		The proof of this result relies heavily on a {\it forward-and-backward} analysis of the event
		\begin{center}
			{\it``An active individual sampled at time $\tau+h$ and another active individual independently sampled at time $\tau$ are of type $\heartsuit$"}.
		\end{center}
		This approach provides a novel strategy for establishing tightness in processes that admit a duality relation. The proofs of Lemma \ref{lem:tight_z}, Lemma \ref{lem:dual_tight} and Theorem \ref{thm:weak_conv} are provided in Subsection \ref{app: pf wc}.

		\section{Conditioned process}\label{4. Conditioned process}
		
		In this section, we point out how to use the lookdown construction of the $N$-$(\Lambda,\xi,\mathrm{M})$-Seed-Bank-Moran Model and the convergence result stated adobe in Theorem \ref{thm:weak_conv} to prove Theorem \ref{thm:fixation}. The following lemma establishes that  type $\heartsuit$ fixes in a $N$-$\Lambda$-Seed-Bank-lookdown model if and only if $X_1^{N,1}(0) =\heartsuit$ almost surely. Since events $\{X_1^{N,1}(0)=\heartsuit\}$ coincide for all $N \in \NN$, we simply write $\{X_1^{1}(0)=\heartsuit\}$.
		\begin{lem}[Fixation equivalence]
			\label{lem:fix_levels}
			The events 
			\[\left\{X_1^{1}(0)=\heartsuit\right\},\quad\left\{\lim_{t\rightarrow \infty} Z_1^N (t)+Z_2^N (t)=1\quad \forall N\in\NN \right\}\quad\text{and}\quad \left\{\lim_{t\rightarrow \infty} Z_1 (t)+Z_2 (t)=1\right\}\]
			are identical up to a null set.
		\end{lem} 	
		\begin{proof}
			For every~$N\in \NN$, if $\omega \in \left\{\lim\limits_{s\rightarrow \infty} Z_1^N(t)+Z_2^N (t) = 1\right\}$, then there exist a time $t\geq 0$ in which $X_1^1(t)(\omega)=\heartsuit$. Since the type on the first line is given in at time $0$, we deduce that $X_1^1(0)(\omega) = \heartsuit$. Thus,
			\begin{align*}
				\left\{\lim\limits_{s\rightarrow \infty} Z_1^N(t)+Z_2^N(t) = 1\right\} \subseteq \left\{X_1^1(t) = \heartsuit \text{ for some } t\geq 0 \right\} = \left\{X_1^1(0)=\heartsuit\right\}.
			\end{align*}
			Since the lookdown particle systems are embedded one into the other according to the size $N$ a direct comparison of fixation allow us to conclude that 
			\begin{align*}
				\left\{\lim\limits_{t\rightarrow \infty} Z_1 (t)+Z_2 (t) = 1\right\}\subseteq \left\{\lim\limits_{t\rightarrow \infty} Z_1^N (t)+Z_2^N(t) = 1\right\} \subseteq \left\{X_1^1(0) = \heartsuit\right\}.
			\end{align*}
			Let us now demonstrate the converse. Let $A_t(i,T)$ be the ancestral line at time $t$ of the level $i$ sampled at time $T$. Define the stopping time $\tau_i$, which corresponds to the first time when the ancestral line started in the level $i$ at time $T$ reaches level $1$, that is			
			\begin{align*}
				\tau_{i} \coloneqq \inf\{t \geq 0\,:\, A_t(i,T) = (1,T-t)  \}, \qquad i \in \NN.
			\end{align*}
			Note that $\tau_i$ is finite almost surely for all fix levels $i$, and 
			\begin{align*}
				\p[\tau_{1}>t] = 0, \qquad t>0.
			\end{align*}
			Condition on the event that $\{X_1^{1}(0)=\heartsuit\}$, the following bound arises 
			\begin{align*}
				\frac{1}{N}\sum_{i=1}^{N}\1_{\{\tau_{i}\leq t\}} \leq \left(Z_1^N(t)+Z_2^N(t)\right)\Big|_{\{X_1(0)=\heartsuit\}}.
			\end{align*}
			Therefore  
			\begin{align*}
				\E\left[\lim_{t\rightarrow \infty} 1 \right.&\left. -\left(Z_1(t)+Z_2(t)\right)\Big|X_1(0)=\heartsuit \right]\\
				&= \lim_{t\rightarrow \infty} \E\left[ 1-\left(Z_1(t)+Z_2(t)\right)\Big|X_1(0)=\heartsuit \right] \\
				& = \lim_{t\rightarrow \infty}\E\left[\lim_{N\rightarrow \infty} 1-\left(Z^N_1(t)+Z^N_2(t)\right)\Big|\{X_1(0)=\heartsuit\} \right]\\
				\overset{Fatou}&{\leq}\lim_{t\rightarrow\infty} \liminf_{N\rightarrow\infty}\E\left[1-\left(Z^N_1(t)+Z^N_2(t)\right)\Big|\{X_1(0)=\heartsuit\} \right] \\
				& \leq \lim_{t\rightarrow\infty} \liminf_{N\rightarrow\infty}\E\left[\frac{1}{N}\sum_{i=2}^{N}\1_{\{\tau_{i}\geq t\}} \right]\\
				&= \lim_{t\rightarrow\infty} \liminf_{N\rightarrow\infty} \frac{N-1}{N}\p\left[\tau_{2}> t\right]\\
				& \leq \lim_{t\rightarrow\infty}\p\left[\tau_{2}> t\right]= \p\left[\tau_{2} = \infty\right]= 0.
			\end{align*}
			We exchange the integral and the limit in the first equality by dominated convergence. Since the limit random variable is non negative, we are able to conclude that 
			\[
			\lim_{t\rightarrow \infty} \left(Z_1(t)+Z_2(t)\right)\Big|_{\{X_1(0)=\heartsuit\}}= 1 \qquad \text{ a.s. }
			\]
		\end{proof}
		
		The next lemma establishes the distribution of the state process of the upper levels once the first level is conditioned to be of type $\heartsuit$.
		\begin{lem}
			\label{lem_fixation}
			Consider $\{\widetilde{X}^N(t)\}_{t\geq 0}$ defined as follows: 
			\begin{align*}
				\left(\widetilde{X}_1(t),\ldots,\widetilde{X}_{N-1}(t)\right) \coloneqq \left(X_2(t),\ldots,X_N(t)\right)\big|_{\left\{X_1(0)=\heartsuit\right\}},\quad \forall t\geq 0.
			\end{align*}
			Then, the process $\widetilde{X}^N$ corresponds to a $(N-1)$-$\left(\Lambda,\xi,\mathrm{M}\right)$-Seed-Bank-lookdown Moran model, where $\xi_t =\1_{\left\{X_1^{N,2}(t) = \heartsuit\right\}}$ and $\mathrm{M}(dy)= y^{-1}\Lambda(dy)$.
		\end{lem}
		
		\begin{figure}[h]
			\centering
			\includegraphics[width=1\textwidth]{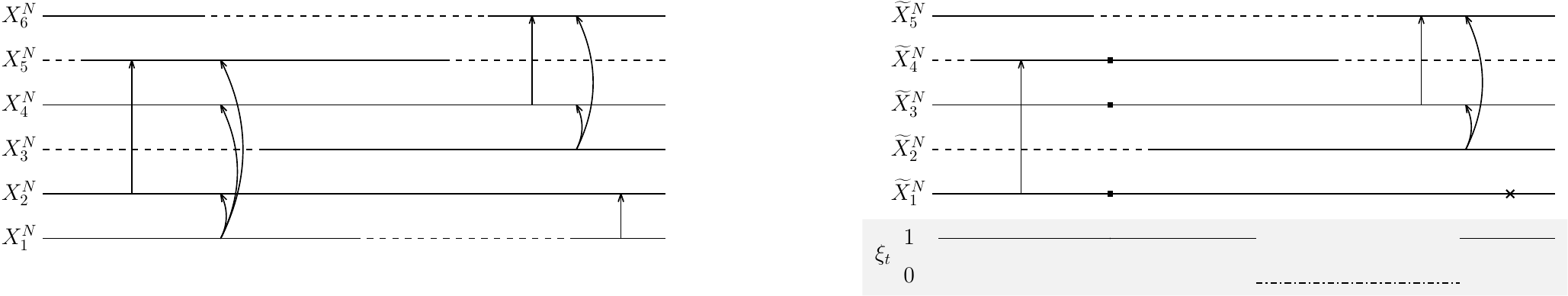}
			\caption{Illustration of the particle system in Lemma \ref{lem_fixation}. The state process of the first level becomes the environment process, given by $\xi_t =\1_{\left\{X_1^{N,2}(t) = \heartsuit\right\}}$. Reproduction events involving the first level became mutations, either single or multiple.}
		\end{figure}
		
		\begin{proof}[Proof Lemma~\ref{lem_fixation}.]
			Reproduction events that do not involve the first line correspond to reproduction events for a $(N-1)$-$\left(\Lambda,\xi,\mathrm{M}\right)$-Seed-Bank-lookdown Moran model. Moreover, reproduction events that involve the first line can be understood as mutation to type $\heartsuit$  since all individuals inherit the first level type in such events, those events are only allowed when the first line is active; therefore, the state process of the first line is the environment of mutations.
		\end{proof}
		The following corollary is a direct implication of Lemma~\ref{lem:fix_levels} and Lemma~\ref{lem_fixation}.
		\begin{coro}
			\label{coro:fixation_freq}
			Let $\{\mathbf{Z}^N(t)\}_{t\geq 0}$ the frequency process of a $N$-$\Lambda$-Seed-Bank-lookdown Moran model. Define 
			\begin{align*}
				\mathbf{Z}^N(t)\Big|_{\left\{\lim\limits_{s\rightarrow \infty}Z_1^N(s)+Z_2^N(s)=1 \right\}} = \frac{1}{N}\left(\xi_t,1-\xi_t,\xi_t \right)+\frac{N-1}{N}\mathbf{\widetilde{Z}}^{N-1}(t),\quad\forall t\geq0 
			\end{align*}
			where $\xi_t = X_1^N(t)$ and $\left(\mathbf{\widetilde{Z}}^{N-1}(t),\xi_t\right)$ correspond to the frequency and the environmental processes of a $(N-1)$-$\left(\Lambda,\xi,\mathrm{M}\right)$-Seed-Bank-lookdown Moran model at time $t$, respectively.
		\end{coro}
		\section{Proofs}
		\subsection{Proof Theorem \ref{thm:Strong_sol}}
		\label{proof:Strong_sol}
		\begin{proof}[Proof Theorem \ref{thm:Strong_sol}]			
		This proof follows Yamada-Watanabe type of argument to ensure the existence and uniqueness of a strong solution. By Theorem \ref{thm:weak_conv} the system of stochastic equations with jumps $\eqref{eq:sys_sdbkm}$ accepts a weak solution. By \cite[Theorem 1.2]{YW_Jumps} is sufficient to show pathwise uniqueness to guarantee the existence of a unique strong solution. The following pathwise uniqueness proof is an adaption of ideas from \cite[Theorem 2.1]{Dawson&Li} and \cite[Theorem 3.2]{F&P}. Let $\left(\widehat{Z},\widehat{\xi}\right)$ and $\left(\overline{Z},\overline{\xi}\right)$ two weak solutions to $\eqref{eq:sys_sdbkm}$ with the same initial condition $(z_1,z_2,z_3)\in D$ and $\widehat{\xi}_0,\overline{\xi}_0\in \{0,1\}$. 
		
		Note first that in the system $\eqref{eq:sys_sdbkm}$ the frequency of active individuals correspond to an ordinary differential equation decoupled from the rest with solution given by
		\begin{align} \label{eq: solz3}
			\widetilde{Z}_3(t)\coloneqq\widehat{Z}_3(t) = \overline{Z}_3(t) = \left(z_3-\frac{\sigma}{\sigma+\alpha}\right)e^{-(\alpha+\sigma)t}+\frac{\sigma}{\sigma+\alpha}.
		\end{align}
		In a similar way, the switching environment evolves independently from the other components as a continuous time Markov chain with two states, so it a has a unique pathwise solution, which implies that $\widetilde{\xi}_t\coloneqq \widehat{\xi}_t=\overline{\xi}_t$ for all $t\geq 0$. 
		
		The next step is to derive a Gronwall inequality for the difference between the weekly solutions for the remaining terms
		\begin{align}
			\zeta_i = \widehat{Z}_i(t)-\overline{Z}_i(t), \text{ for }i\in \{1,2\}.
		\end{align} 
		Let us define $\{a_k\}_{k \in \mathbb{N}}$ by $ a_k= e^{-\frac{k(k+1)}{2}}$. Note that $\{a_k\}_{k \in \mathbb{N}}$ is a decreasing sequence which tends to $0$ as $k$ goes to infinity, and $\int\limits_{a_{k}}^{a_{k-1}}z^{-1}dz = k$. Let $z\mapsto \psi_k(z)$ a non-negative function with support on $(a_{k},a_{k-1})$ such that $\int\limits_{a_k}^{a_{k-1}}\psi_k(z)dz=1$, and $0\leq \psi_k(z)\leq2k^{-1}z^{-1}$ for $z \in (a_k,a_{k-1})$. For $k\in \NN$, define $\varphi_k:[0,1]\rightarrow[0,1]$ the non-negative and twice continuously differentiable function 
		\begin{align*}
			\varphi_k(z)\coloneqq \int_{0}^{|z|}\int_{0}^{y}\psi_k(x)dxdy.
		\end{align*}     
		The sequence $\{\varphi_k\}_{k=1}^\infty$ satisfies the following properties:
		\begin{enumerate}[label=\roman*]
			\item $\varphi_k(z)\rightarrow |z|$ non-decreasing as $k\rightarrow \infty$.
			\item $0\leq \varphi_k'(z)\leq 1$ for $z\geq 0$ and $-1\leq \varphi_k(z)\leq 0$ for $z\leq 0$.
			\item $\varphi_k''(z)=\psi_k(|z|) \leq 2k^{-1}|z|^{-1}$.
		\end{enumerate}	  
		Using Ito's formula for $\varphi_k(\zeta^1_t)$  we have that
		\begin{align*}
			\varphi_k\left(\zeta^1_t\right)& = \int_{0}^{t}\varphi'_k\left(\zeta^1_{s}\right)\left(-\sigma\zeta^2_{s}+\alpha\zeta^1_s\right)ds-\int_{0}^{t}b\xi_s\varphi_k'(\zeta^1_s)\zeta_s^1ds\\
			&+\int_{0}^{t}\varphi_k'(\zeta_s^1)\left(\sqrt{a\widehat{Z}_1(s)\left(\widetilde{Z}_3(s)-\widehat{Z}_1(s)\right)}-\sqrt{a\overline{Z}_1(s)\left(\widetilde{Z}_3(s)-\overline{Z}_1(s)\right)}\right)dB_s\\
			&+\frac{1}{2}\int_{0}^{t}\varphi_k''(\zeta_s^1)\left(\sqrt{a\widehat{Z}_1(s)\left(\widetilde{Z}_3(s)-\widehat{Z}_1(s)\right)}-\sqrt{a\overline{Z}_1(s)\left(\widetilde{Z}_3(s)-\overline{Z}_1(s)\right)}\right)^2ds\\
			&+\int_{0}^{t}\int_{0}^{1}\int_{0}^{1}r\left[\widetilde{Z}_3(s_{\shortminus})-\widehat{Z}_1(s_{\shortminus})\right]\1_{\left\{u\leq\frac{\widehat{Z}_1(s)}{\widetilde{Z}_3(s)}\right\}}\left(\varphi_k(\zeta_s^1+r)-\varphi_k(\zeta_s^1)\right)\widetilde{N}(ds,dr,du) \\
			&-\int_{0}^{t}\int_{0}^{1}\int_{0}^{1}r\widehat{Z}_1(s_{\shortminus})\1_{\left\{u>\frac{\widehat{Z}_1(s)}{\widetilde{Z}_3(s)}\right\}}\left(\varphi_k(\zeta_s^1+r)-\varphi_k(\zeta_s^1)\right)\widetilde{N}(ds,dr,du) \\
			&-\int_{0}^{t}\int_{0}^{1}\int_{0}^{1}r\left[\widetilde{Z}_3(s_{\shortminus})-\overline{Z}_1(s_{\shortminus})\right]\1_{\left\{u\leq\frac{\overline{Z}_1(s)}{\widetilde{Z}_3(s)}\right\}}\left(\varphi_k(\zeta_s^1+r)-\varphi_k(\zeta_s^1)\right)\widetilde{N}(ds,dr,du)\\
			&+\int_{0}^{t}\int_{0}^{1}\int_{0}^{1}r\overline{Z}_1(s_{\shortminus})\1_{\left\{u>\frac{\overline{Z}_1(s)}{\widetilde{Z}_3(s)}\right\}}\left(\varphi_k(\zeta_s^1+r)-\varphi_k(\zeta_s^1)\right)\widetilde{N}(ds,dr,du)\\
			&+\int_0^t\int_0^1 r\xi_s\left(\zeta^3_s- \zeta^1_s\right)\left[\varphi_k(\zeta^1_s+r)-\varphi_k(\zeta_s^1)\right]\widehat{N}(ds,dr). 
		\end{align*}
		Note that the integrand in the third term on the right can be bounded by
		
		\begin{align*}
			\left(\sqrt{a\widehat{Z}_1(s)\left(\widetilde{Z}_3(s)-\widehat{Z}_1(s)\right)}-\sqrt{a\overline{Z}_1(s)\left(\widetilde{Z}_3(s)-\overline{Z}_1(s)\right)}\right)^2\leq 2a\widetilde{Z}_3(t)|\zeta^1_t|\leq 2a|\zeta^1_t|,
		\end{align*}
		therefore, it follows from  $\psi_k$ that such a term is bounded by  $\tfrac{2ta}{k}$, and that it goes to zero as $k$ goes to infinity.
		
		To bound the mean of the other terms on the right hand side, note that the second term on the right hand side is a martingale, and we can compensate the integrals with respect to the point Poisson processes to obtain also martingales. Moreover, note that the if we consider the related intensity of the fourth and fifth lines, we have that 
		\begin{align*}
			&\int_{0}^{1}r\left[\left(\widetilde{Z}_3(s_\shortminus)-\overline{Z}_1(s_\shortminus)\right)\1_{\left\{u\leq\frac{\overline{Z}_1(s)}{\widetilde{Z}_3(s)}\right\}}-\overline{Z}_1(s_\shortminus)\1_{\left\{u>\frac{\overline{Z}_1(s)}{\widetilde{Z}_3(s)}\right\}}\right]\left(\varphi_k(\zeta_s^1+r)-\varphi_k(\zeta_s^1)\right)du \\
			& = r\left[\left(\widetilde{Z}_3(s_\shortminus)-\overline{Z}_1(s_\shortminus)\right)\frac{\overline{Z}_1(s)}{\widetilde{Z}_3(s)}-\overline{Z}_1(s_\shortminus)\left(1-\frac{\overline{Z}_1(s)}{\widetilde{Z}_3(s)}\right)\right]\left(\varphi_k(\zeta_s^1+r)-\varphi_k(\zeta_s^1)\right)=0
		\end{align*}
		and the same calculation holds for $\widehat{Z}$. On the other hand, the integral with respect to the intensity measure of $\widehat{N}$ process shows that
		
		\begin{align*}
			\left|\int_0^t\int_0^1 r\xi_s \zeta^1_s\left[\varphi_k(\zeta^1_s+r)\varphi_k(\zeta_s^1)\right] \frac{\Lambda_0(\dd r) }{r} \dd s \right| & \leq 2\Lambda_0([0,1])\int_0^t \left|\zeta_s^1\right|ds.
		\end{align*}  
		Hence, if we take expectation and let $k$ go to infinity, we deduce from the monotone convergence theorem that
		\begin{align*}
			\E\left[|\zeta^i_t|\right]\leq C \int_{0}^{t}\E\left[|\zeta^1_s|\right]+\E\left[|\zeta^2_s|\right]ds, \text{ for }i\in \{1,2\}
		\end{align*} 
		for some constant $C>0$, which implies the pathwise uniqueness for equation \eqref{eq:sys_sdbkm}.
		\end{proof}
		\subsection{Proof of Theorem~\ref{thm:weak_conv}} \label{app: pf wc}
		Recall that a sequence of stochastic processes $\left\{W^N\right\}_{N\in\NN}$ is tight if satisfies the {\it Aldous tightness criterion} \cite[Lemma 23.12]{Kallenberg}, that is if for any $t\geq 0$,
		\begin{align}
			\label{aldous_eq}
			\lim_{\delta \rightarrow 0} \limsup_{N\rightarrow \infty} \sup_{\tau \leq t} \sup_{h\in[0,\delta]} \E_{w}\left[\left|W(\tau)-W(\tau+h)\right| \right]=0, \qquad w\in[0,1].
		\end{align}
		\begin{proof}[Proof Lemma \ref{lem:dual_tight}]
			We first establish the following bound for the sampling duality function $H^N$ defined in \eqref{eq: defH}, which will be useful in the following.  
			\begin{align}
				\label{eq:sampling_bound}
				H^N(\mathbf{z},1,0)^2 \leq H^N(\mathbf{z},2,0)+\frac{\1_{\{z_3>0\}}}{z_3(N-1)}, \qquad \mathbf{z}\in D^N. 
			\end{align}
			Note that the inequality is direct if $z_3=0$, if $z_3 \in \{\tfrac{1}{N},\tfrac{2}{N},\dots 1\}$ then $Nz_3-1\leq (N-1)z_3$ thus 
			\begin{align*}
				H^N(\mathbf{z},2,0) = \frac{z_1(Nz_1-1)}{z_3(Nz_3-1)} \geq \frac{z_1(z_1(N-1)+z_1-1)}{z_3^2(N-1)} \geq \left(\frac{z_1}{z_3}\right)^2-\frac{1}{z_{3}(N-1)}.
			\end{align*} 
			Therefore, using \eqref{eq:sampling_bound} and the Cauchy inequality we deduce that 
			\begin{align*}
				&\E_{\mathbf{z}}\left[\left|H^N(\mathbf{Z}^N(\tau+h),1,0)-H^N(\mathbf{Z}^N(\tau),1,0)\right|\right]^2\\
				& \leq \E_{\mathbf{z}}\left[\left(H^N(\mathbf{Z}^N(\tau+h),1,0)-H^N(\mathbf{Z}^N(\tau),1,0)\right)^2\right]\\
				& = \E_{\mathbf{z}}\left[H^N(\mathbf{Z}^N(\tau+h),1,0)^2\right]+\E_{\mathbf{z}}\left[H^N(\mathbf{Z}^N(\tau),1,0)^2 \right]\\
				& \quad - 2\E_{\mathbf{z}}\left[H^N(\mathbf{Z}^N(\tau+h),1,0)H^N(\mathbf{Z}^N(\tau),1,0)\right]\\
				&\leq \E_{\mathbf{z}}\left[H^N(\mathbf{Z}^N(\tau+h),2,0)+\frac{\1_{\{\mathbf{Z}_3^N(\tau+h)>0\}}}{\mathbf{Z}_3^N(\tau+h)(N-1)}\right]\\
				& \quad +\E_{\mathbf{z}}\left[H^N(\mathbf{Z}^N(\tau),2,0) +\frac{\1_{\{\mathbf{Z}^N_3(\tau)>0\}}}{\mathbf{Z}_3^N(\tau)(N-1)}\right]\\
				&\quad -2 \E_{\mathbf{z}}\left[H^N(\mathbf{Z}^N(\tau+h),1,0)H^N(\mathbf{Z}^N(\tau),1,0)\right].
			\end{align*}
			We will bound each term, for the first one notice that inequality \eqref{eq:sampling_bound} already give us a bound moreover using duality we can deduce that
			\begin{equation}\label{eq : firsterm}
				\begin{split}
					\E_{\mathbf{z}}\left[H^N(\mathbf{Z}^N(\tau+h),2,0) \right]
					&=\E_{(2,0)}\left[H^N(\mathbf{z}, N_{\tau+h},M_{\tau+h}) \right]\\
					&= \E_{(2,0)}\left[H^N(\mathbf{z}, N_{\tau},M_{\tau})\1_{\left\{(N_{\tau+h},M_{\tau+h})=(N_{\tau},M_{\tau}) \right\}}\right] \\
					&\qquad+ \E_{(2,0)}\left[H^N(\mathbf{z}, N_{\tau+h},M_{\tau+h})\1_{\left\{(N_{\tau+h},M_{\tau+h})\neq(N_{\tau},M_{\tau}) \right\}} \right]\\
					& \leq  \E_{(2,0)}\left[H^N(\mathbf{z}, N_{\tau},M_{\tau})\right]+\p_{(2,0)}\left[(N_{\tau+h},M_{\tau+h})\neq(N_{\tau},M_{\tau})\right],
				\end{split}    
			\end{equation}
			where
			\begin{equation*}
				\begin{split}
					\E_{(2,0)}\left[H^N(\mathbf{z}, N_{\tau},M_{\tau})\right]=\E_{\mathbf{z}}\left[H^N(\mathbf{Z}^N(\tau),2,0) \right].
				\end{split}    
			\end{equation*}
			To study the product term at times $\tau+h$ and $\tau$, we introduce the event 
			\begin{center}
				$B'$ = {\it ``an active individual sampled at time $\tau+h$ and another active individual independently sampled at time $\tau$ are of type $\heartsuit$, given that the frequencies at time $0$ were $\mathbf{z}$."}
			\end{center} 
			Let us now do a forward-and-backward analysis of this event.
			\begin{itemize}
				\item {\bf Forward analysis.} Define the following auxiliary events: 
				\begin{center}
					$B'_{\tau+h}$ = {\it ``An active individual sampled at time $\tau+h$  is of type $\heartsuit$ given that the frequencies at time $0$ were $z$"}
				\end{center} and
				\begin {center}
				$B'_{\tau}$ = {\it ``An active individual sampled at time $\tau$ is of type $\heartsuit$ given that the frequencies at time $0$ were $z$."}	
			\end{center}
			Note that $B' = B'_{\tau}\cap B'_{\tau + h}$ then 
			\begin{align*}
				\E\left[\1_{B'}\right.& \left.\Big| \mathbf{Z}^N(\tau),\mathbf{Z}^N(\tau+h)\right] \\
				&= \E\left[\E\left[\1_{B'_\tau}\1_{B'_{\tau + h}}\Big| \mathbf{Z}^N(\tau),\mathbf{Z}^N(\tau+h),\1_{B'_\tau}\right]\Big| \mathbf{Z}^N(\tau),\mathbf{Z}^N(\tau+h)\right]\\
				& = \E\left[\1_{B'_\tau}\E\left[\1_{B'_{\tau + h}}\Big| \mathbf{Z}^N(\tau),\mathbf{Z}^N(\tau+h),\1_{B'_\tau}\right]\Big| \mathbf{Z}^N(\tau),\mathbf{Z}^N(\tau+h)\right]\\
				& = \E\left[\1_{B'_\tau}\E\left[\1_{B'_{\tau + h}}\Big| \mathbf{Z}^N(\tau+h)\right]\Big| \mathbf{Z}^N(\tau),\mathbf{Z}^N(\tau+h)\right]\\
				& =  \E\left[\1_{B'_\tau}H^N\left(\mathbf{Z}^N(\tau+h),1,0\right)\Big| \mathbf{Z}^N(\tau),\mathbf{Z}^N(\tau+h)\right]\\
				&= H^N\left(\mathbf{Z}^N(\tau+h),1,0\right)\E\left[\E\left[\1_{B'_\tau }\Big| \mathbf{Z}^N(\tau)\right]\Big| \mathbf{Z}^N(\tau),\mathbf{Z}^N(\tau+h)\right]\\
				&= H^N\left(\mathbf{Z}^N(\tau+h),1,0\right)H^N\left(\mathbf{Z}^N(\tau),1,0\right), 
			\end{align*}
			so 
			\begin{equation*}\label{eq: FW analysis}
				\p[B'] = \E\left[H^N\left(\mathbf{Z}^N(\tau+h),1,0\right)H^N\left(\mathbf{Z}^N(\tau),1,0\right)\right]
			\end{equation*}
			\item {\bf Backward analysis.} On the other hand, let $U_1$ and $U_2$ be the levels of the sampled individuals on times $\tau$ and $\tau + h $ respectively, and recall the ancestral line $A_s(i,t)$ at time $s$ of the level $i$ sampled at time $T$.  By the law of total probability
			\begin{align*}
				\p[B']  =& \p\left[B'\Big|A_h(U_{2},\tau+h)=(U_2,\tau)\right]\p\left[A_h(U_2,\tau+h) = (U_2,\tau)\right]\\&+\p\left[B'\Big|A_h(U_2,\tau+h)\neq(U_2,\tau) \right]\p\left[A_h(U_2,\tau+h) \neq (U_2,\tau)\right].
			\end{align*}
			Furthermore, 
			\begin{align*}
				\p&\left[B'\Big|A_h(U_{2},\tau+h)=(U_2,\tau)\right]\\
				& = \p\left[B'\Big|U_1=U_2,A_h(U_{2},\tau+h)=(U_2,\tau)\right]\p\left[U_1=U_2\Big|A_h(U_{2},\tau+h)=(U_2,\tau)\right]\\
				& \quad  +\p\left[B'\Big|U_1=U_2,A_h(U_{2},\tau+h)=(U_2,\tau)\right]\p\left[U_1=U_2\Big|A_h(U_{2},\tau+h)=(U_2,\tau)\right].
			\end{align*} 
			Note that $\left\{B'|U_1\neq U_2,A_h(U_{2},\tau+h)=(U_2,\tau)\right\}$ correspond to the event 
			\begin{center}
				{\it ``Sampling two active individuals independently at time $\tau$ given that the frequencies at time $0$ were $z$"},    
			\end{center}
			so 
			\begin{align*}
				\p\left[B'\Big|U_1\neq U_2,A_h(U_{2},\tau+h)=(U_2,\tau)\right]&= \E_\mathbf{z}\left[H^N(\mathbf{Z}^N(\tau),2,0)\right]\\
				&=\E_{(2,0)}\left[H^N(\mathbf{z},N_{\tau},M_{\tau}) \right].
			\end{align*}				
		\end{itemize}
		By merging both approaches we have that  
		\begin{equation} \label{eq: product term}
			\begin{split}
				\E_{\mathbf{z}}&\left[H^N\left(\mathbf{Z}^N(\tau),1,0\right)H^N\left(\mathbf{Z}^N(\tau+h),1,0\right)\right]\\
				&\geq \E_{(2,0)}\left[H^N(\mathbf{z},N_{\tau},M_{\tau}) \right]\p\left[U_1\neq U_2,A_h(U_{2},\tau+h)=(U_2,\tau)\right].
			\end{split}
		\end{equation}  
		If we add up inequalities \eqref{eq : firsterm} and \eqref{eq: product term}, we can deduce that 
		\begin{align*}
			\E_{\mathbf{z}}&\left[\left(H^N(\mathbf{Z}^N(\tau+h),1,0)-H^N(\mathbf{Z}^N(\tau),1,0)\right)^2\right]\\
			& \leq   2\left(1-\p\left[U_1\neq U_2,A_h(U_{2},\tau+h)=(U_2,\tau)\right] \right) +\p_{2,0}\left[(N_{\tau+h},M_{\tau+h})\neq(N_{\tau},M_{\tau})\right]\\
			&\quad  + \E_{z_3}\left[\frac{\1_{\left\{Z_3^N(t)>0\right\}}}{(N-1)Z_3^N(\tau+h) }\right] + \E_{z_3}\left[\frac{\1_{\left\{Z_3^N(\tau)>0\right\}}}{(N-1)Z_3^N(\tau) }\right]\\
			& \leq   2\p\left[U_1= U_2\right]+ 2\p\left[A_h(U_{2},\tau+h)\neq(U_2,\tau)\right]  +\p_{2,0}\left[(N_{\tau+h},M_{\tau+h})\neq(N_{\tau},M_{\tau})\right]\\
			& \quad + \E_{z_3}\left[\frac{\1_{\left\{Z_3^N(t)>0\right\}}}{(N-1)Z_3^N(\tau+h) }\right] + \E_{z_3}\left[\frac{\1_{\left\{Z_3^N(\tau)>0\right\}}}{(N-1)Z_3^N(\tau) }\right]
		\end{align*}
		which guarantees \eqref{aldous_eq}. The tightness of the other sequence is derived in an a similar way.
		\end{proof}
		\begin{proof}[Proof Lemma \ref{lem:tight_z}] Recall the definition of the frequency of active individuals \begin{align*}
			Z_3^N(t)\coloneqq \frac{1}{N}\sum\limits_{i=1}^N\sum\limits_{e\in E} \1_{\{\mathbf{Y}_i^N(t)=(e,\mathrm{a})\}} = \frac{1}{N} \sum_{i=1}^N 
			\1_{\left\{Y_{i}^{N,2}(t)=a\right\}}.
		\end{align*} To check the tightness of the sequence  $\left\{ \left(Z_3^N(t)\right)_{t\geq 0}\right\} _{N\in \NN}$ notice that for $\tau\leq t$ and $h\in [0,\delta]$
		\begin{align*}
			\E\left[ \left| Z_3^N(\tau+h)-Z_3^N(\tau) \right| \right]& = \E\left[\left|\frac{1}{N} \sum_{i=1}^N 
			\1_{\left\{Y_{i}^{N,2}(\tau+h)=a\right\}}-\frac{1}{N}\sum_{i=1}^N \1_{\left\{Y_{i}^{N,2}(\tau)=a\right\}} \right|\right]\\
			&\leq \frac{1}{N}\sum_{i=1}^N\E\left[\left|\1_{\left\{Y_{i}^{N,2}(\tau+h)=a\right\}}-\1_{\left\{Y_{i}^{N,2}(\tau)=a\right\}}\right|\right]\\
			&\leq \E\left[\left|\1_{\left\{Y_{1}^{N,2}(\tau+h)=a\right\}}-\1_{\left\{Y_{1}^{N,2}(\tau)=a\right\}}\right|\right]\\
			&=\p\left[\left|\1_{\left\{Y_{1}^{N,2}(\tau+h)=a\right\}}-\1_{\left\{Y_{1}^{N,2}(\tau)=a\right\}}\right|\neq 0 \right]\\
			& \leq \sigma \delta +\alpha \delta +O\left(\delta\right) 
		\end{align*} 
		therefore the Aldous tightness criterium is satisfied for $\{(Z_3^N(t))_{t\geq0}\}_{N\in\NN}$. Aldous criterion is deduced for the other two entries by Lemma \ref{lem:dual_tight} and the following inequality 
		\begin{equation*}
			\begin{split}
				\left|Z_1^N(\tau+h)-Z_1^N(\tau)\right| &\leq \left|Z_3^N(\tau+h)-Z_3^N(\tau)\right|\\
				& \quad +\left|H^N(\mathbf{Z}^N(\tau+h),1,0)-H^N(\mathbf{Z}^N(\tau),1,0)\right|.
			\end{split}
		\end{equation*}		
		\end{proof}
		\begin{proof}[Proof Theorem \ref{thm:weak_conv}]

	First, note that by Remark \ref{rem:conv_generator} and the Markov property, it follows that the finite-dimensional distributions of the sequence  $\left\{\left\{\left(\mathbf{Z}^N(t),\xi(t)\right)\right\}_{t\geq 0}\right\}_{n\in N}$ converge to those of the Markov process characterized by the generator in Remark \ref{rem:conv_generator}. By Lemma \ref{lem:tight_z}, the sequence  $\left\{\left\{\left(\mathbf{Z}^N(t),\xi(t)\right)\right\}_{t\geq 0}\right\}_{n\in N}$ is tight then by \cite[Theorem 23.2]{Kallenberg}, it is also relatively compact in distribution. Therefore any subsequential limit ${(Z(t), \xi(t))}_{t \ge 0}$ solves the martingale problem associated with the generator in \eqref{gen::freq_l_m}.  Finally, by \cite[Theorem 23.9]{Kallenberg}, we obtain weak convergence in $D(\mathbb{R}^+, D \otimes \{0,1\})$. 
		\end{proof}
		
		\subsection{Proof of Theorem  \ref{thm:fixation}} \label{pf: fixation} 
		\begin{proof}[Proof Theorem \ref{thm:fixation}] This proof follows from the tightness of the family $\widetilde{\mathbf{Z}}$, by Lemma \ref{lem:tight_z}, and the finite dimensional convergence.  
		The proof could be visualized in the following diagram 
		\begin{center}
			\begin{tikzcd}[row sep=4em, column sep=6em]
				\left(\mathbf{Z},\xi\right)\Big|_{\left\{\lim\limits_{s\to\infty} Z_1(s)+Z_2(s)=1 \right\}}
				\arrow[r,leftrightarrow,"\text{Theorem \ref{thm:fixation}}"] 
				& \left(\widetilde{\mathbf{Z}},\xi\right) \\
				\left(\mathbf{Z}^N,\xi\right)\Big|_{\left\{\lim\limits_{s\to\infty} Z_1(s)+Z_2(s)=1 \right\}} 
				\arrow[r,leftrightarrow,"\text{Corollary \ref{coro:fixation_freq}}"'] 
				\arrow[u,"\text{Theorem \ref{thm:weak_conv}}"] 
				& \left(\widetilde{\mathbf{Z}}^N,\xi\right)
				\arrow[u,"\text{Theorem \ref{thm:weak_conv}}"']
			\end{tikzcd}
		\end{center}
		
		Let us show that both processes have the same finite dimensional distributions. To this end, recall that $\xi_t = \1_{\{X_1^{N,2}(t) = \heartsuit \}}$. 
		
		Let $\{t_i\}_{i=1}^n\subset \R^+_0$, and $\{B_i\}_{i=1}^n\subset \mathcal{B}\left(D\times \{0,1\}\right)$, then 
		\begin{align*}
			\p\left[\mathop{\cap}_{i=1}^n\left\{ \left(\mathbf{Z}(t_i),\xi_{t_i} \right)\in B_i \right.\right.&\left.\left. \right\}\Big| \lim_{s\rightarrow \infty} Z_1(s)+Z_2(s) = 1  \right] \\ &\overset{\text{Lemma \ref{lem:fix_levels}}}{=}  \p\left[\mathop{\cap}_{i=1}^n\left\{ \left(\mathbf{Z}(t_i),\xi_{t_i} \right)\in B_i  \right\}\left|  X_1^{E}(0) = \heartsuit \right. \right] \\
			&\overset{\text{Theorem \ref{thm:weak_conv}}}{=} \lim_{N\rightarrow\infty} \p\left[\mathop{\cap}_{i=1}^n\left\{ \left(\mathbf{Z}^N(t_i),\xi_{t_i} \right)\in B_i  \right\}\left|  X_1^{E}(0) = \heartsuit \right. \right] \\
			&\overset{\text{Corollary \ref{coro:fixation_freq}}}{=} \lim_{N\rightarrow\infty} \p\left[\mathop{\cap}_{i=1}^n\left\{ \left(\frac{1}{N}\left(\xi_{t_i},1-\xi_{t_i},\xi_{t_i} \right)+\frac{N-1}{N}\widetilde{\mathbf{Z}}^{N-1}({t_i}),\xi\right)  \in B_i  \right\} \right] \\
			&\overset{\text{Theorem \ref{thm:weak_conv}}}{=} \p\left[\mathop{\cap}_{i=1}^n\left\{ \left(\widetilde{\mathbf{Z}}(t_i),\xi_{t_i} \right)\in B_i  \right\}\right].
		\end{align*}
		\end{proof}
		
		\section*{Acknowledgments}
		The authors are grateful to Sebastian Hummel for his comments on an early version of this paper.
		
		\bibliographystyle{plain} 
		\bibliography{LiteratureLWFcond}
	\end{document}